\documentclass[12pt,twoside]{amsart}
\usepackage{amsmath}
\usepackage{amssymb}
\usepackage{amscd}
\usepackage{xypic}
\newmuskip\pFqmuskip

\usepackage{amsthm}
\usepackage{graphicx}
\usepackage{enumerate}
\usepackage{graphicx}
\usepackage{tikz-cd}
\usepackage[new]{old-arrows}
\usepackage[cmtip,all]{xy}
\usepackage{amssymb}

\usepackage{graphicx}
\usepackage{tikz-cd}
\usepackage[new]{old-arrows}
\usepackage[cmtip,all]{xy}

\usepackage{amsmath, amsthm, amssymb}

\newtheorem{theorem}{Theorem}[section]
\newtheorem{proposition}{Proposition}[section]
\newtheorem{lemma}[theorem]{Lemma}

\theoremstyle{definition}
\newtheorem{definition}[theorem]{Definition}

\newtheorem{corollary}{Corollary}[theorem]

\newtheorem{remark}[theorem]{Remark}

\numberwithin{equation}{section}

\begin{document}

\title{Graded linearity of Stanley-Reisner ring of broken circuit complexes}

\email{mrahmati@cimat.mx (M. Reza Rahmati), gflores@cio.mx (G. Flores)}%

\author{Mohammad Reza Rahmati, Gerardo Flores}
\address{}
\curraddr{}
\email{mrahmati@cimat.mx (M. Reza Rahmati), gflores@cio.mx (G. Flores)}
\thanks{}


\subjclass[2010]{ }

\date{}

\dedicatory{}

\begin{abstract}
This paper introduces two new notions of graded linear resolution and graded linear quotients, which generalize the concepts of linear resolution property and linear quotient for modules over the polynomial ring $A=k[x_1, \dots ,x_n]$. Besides, we compare graded linearity with componentwise linearity in general. For modules minimally generated by a regular sequence in a maximal ideal of $A$, we show that graded linear quotients imply graded linear resolution property for the colon ideals. On the other hand, we provide specific characterizations of graded linear resolution property for the Stanley-Reisner ring of broken circuit complexes and generalize several results of \cite{RV} on the decomposition of matroids into the direct sum of uniform matroids. Specifically, we show that the matroid can be stratified such that each strata has a decomposition into uniform matroids. We also present analogs of our results for the Orlik-Terao ideal of hyperplane arrangements which are translations of the corresponding results on matroids.
\end{abstract}

\maketitle


\section{Introduction}
We consider finitely generated graded modules over the polynomial ring, $A=k[x_1, \dots ,x_n]$ where $k$ is an infinite field. One of the essential facts about the finitely generated rings over $A$ is the Hilbert syzygy theorem \cite{P}, stating that any such module has a graded minimal $A$-resolution by free modules as 
\begin{equation}
\dots \longrightarrow \bigoplus_jA(-j)^{\beta_{p,j}} \longrightarrow \ \ \dots \ \ \longrightarrow \bigoplus_jA(-j)^{\beta_{0j}} \longrightarrow M  \longrightarrow 0   ,
\end{equation}
where the graded Betti numbers $\beta_{ij}$ of the graded module $M$ are the exponents appearing in a graded minimal free resolution. Another way to express it is $\beta_{ij}=\dim_k Tor_i(M,k)_j$. One can arrange the Betti numbers in a table by assigning $\beta_{i,i+j}$ in the $(i,j)$ entry, called \textit{the Betti table}. The graded $A$-module $M$ is said to have a linear resolution if its non-zero Betti numbers in minimal $A$-resolution appear just in one line of entries, successively in the table. In special cases, one can characterize this property by the combinatorial data associated with these modules. We generalize this concept by mentioning finitely generated graded modules where their non-zero Betti numbers appear successively in several lines. Then, an important question is the combinatorial identification of the above property. A parallel concept is that of having linear quotients. This notion is defined in terms of the colon ideals of the module $M$. It is well known that having a linear quotient implies the linear resolution property for colon ideals, \cite{HH}. 

In this text we also generalize the concept of quotient linearity, and we define the graded linear quotient property. We compare it with graded linear property and prove the analogous theorem in the graded case. Our proof uses a standard inductive technique of passing to the long exact Tor-sequence associated with short exact sequence on colon ideals of $M$. This generalization aims to study the notion of component-wise linearity, which plays a crucial role in commutative algebra. We show that graded linearity is a weaker property than component-wise linearity.  

For Cohen-Macaulay graded ideals $I$ in $A$, the property of having a linear resolution is characterized by the combinatorial data of the ideal such as its Hilbert polynomial. In this case, the Hilbert polynomial can be summarized in a single binomial coefficient. An interesting example is the Stanley-Reisner ideal $I_{<}(X)$ of the Broken circuit complex of matroid $X$. The Stanley-Reisner ideal $I_<(X)$ is generated by monomials on the broken circuits of $X$. The broken circuit complex of an ordered matroid $X$ is a Shellable complex where it shows the Stanley-Reisner ring $A/I_<(X)$ is Cohen-Macaulay \cite{P}. The linearity property, in this case, can be translated into a combinatorial decomposition of the matroid in terms of uniform matroids, [see \cite{RV}] in the form
\begin{equation} \label{eq:1}
X =U_{s,n-r+s} \bigoplus U_{r-s,r-s},
\end{equation} 
where matroid $U_{p,n}$ is the uniform matroid on $n$-element, whose independent sets are subsets having at most $p$ element. The proof of the Theorem uses the assertions of an external bound for the $f$-vectors of a $r-1$ dimensional shellable complex on the vector set $S$ of size $|S|=n$ in the form
\begin{equation} \label{eq:2}
f_{k-1}\geq \sum_{i=0}^{s-1} {n-r+i-1 \choose i}  {r-i  \choose k-i}, \qquad k=0, \dots ,r  
\end{equation}
when all $s$-element subsets of $S$ belong to $\Delta$. The decomposition in \eqref{eq:1} holds when the equality holds in \eqref{eq:2}, \cite{B}. 

A significant attempt in this work is to find analogous characterizations of graded linearity for these ideals. One also asks how the decomposition of the matroid can be extended. In other words, we investigate what the corresponding operation in the category of matroids is. To answer this question, we introduce the notion of stratification (filtration) of matroids by submatroids. Stratification of a matroid $X$ is a decreasing filtration sequence,
\begin{equation} \label{eq:stratify}
X=X_0 \supset X_1 \supset \dots \supset X_d, \qquad S_j=X_j \smallsetminus X_{j+1}
\end{equation} 
by submatroids $X_j$, where we call $X_j \smallsetminus X_{j+1}$ the $j$-th strata of the stratification. If $X$ is a rank $r$ loopless matroid on $[n]$ we establish a new inequality
\begin{equation}
I_{j-1} \geq \sum_{i=0}^{c-1}\sum_{l=1}^dc_l^{(S)}{n-r+l-1 \choose l}{r-i \choose j-i} 
\end{equation}
generalizing \eqref{eq:2}, and the equality holds if and only if the ideal $I_{<}(X)$ has graded linear resolution. In that case, we first establish that the Hilbert function of $I_<(X)$ has the form 
\begin{equation}
H(\frac{A}{I_<(X)},s)=c_0{s+q-1 \choose s}+c_1{s+q-2 \choose s}+ \cdots +c_{d-1}{s+q-d \choose s}
\end{equation}
for some $d$. The coefficients $c_i=c_i(I_<(X))$ can be easily calculated, (called Hilbert coefficient of $I_<(X)$). In this case there exists a stratification as in \eqref{eq:stratify} such that  
\begin{equation}
S_j= U_{r_j-s_j,r_j-s_j} \bigoplus U_{s_j,n_j-r_j+s_j}
\end{equation}
where $s_j \leq r_j$. We have $n=\Sigma_j n_j, \ r=\Sigma_j r_j$. 

To an arrangement of hyperplanes in a vector space, one can associate a matroid, where the independent sets correspond to the independent hyperplanes. In this way, some properties for matroids can be translated into the language of hyperplane arrangements. The properties above can be applied to certain graded ideals associated with a hyperplane arrangement. In particular, they apply to the Orlik-Solomon ideal and also the Orlik-Terao ideal of the arrangements. As an application, we provide similar assertions for graded linearity of the Orlik-Terao ideal of hyperplane arrangements, which translates the above results to this area. We discuss the Koszul property of the Stanley-Reisner ring of the broken circuit complex. We define an extension of Koszulness, namely graded Koszulness, and characterize this concept via the decomposition of the associated matroid. 

Another important concept is the entire intersection property of Orlik-Solomon algebra and the Orlik-Terao ideal of hyperplane arrangements, \cite{CNR}, \cite{RV}, \cite{F}, \cite{RVi}. There appears several characterizations of this property in \cite{RV} concerning the linearity of the corresponding ideals. We express the generalizations of these theorems to the graded linearity as a natural application to our results. We also present an application to graphs and their cycle matroids. Specifically, we discuss the entire intersection property of the facet ideal of $r$-cyclic graph $G_{n,r}$. We provide an alternative for the proof of the Cohen-Macaulayness of the facet ideal of $G_{n,r}$. The properties of the facet ideal $I_F(\Delta)$ has been discussed in \cite{AR}, \cite{Fa}, \cite{Po}, \cite{Sta}, \cite{AKR}. Our method gives a simple proof of the Cohen-Macaulayness of the facet ideal of $G_{n,r}$. This property was proved in \cite{AR} by direct methods. Our proof is based on the criteria characterizing the complete intersection properties on the cycle matroid of a graph.
\subsection{Related works}
This work is based on the results of \cite{NR, RV} about several combinatorial characterizations of linear resolution property for the Stanley-Reisner ring of broken circuit complexes. The main theorem that is proven in that work is that of characterization for the decomposition of the broken circuit complex of a matroid when the associated Stanley-Reiner ring has a linear resolution. The authors also express the analogs theorems on the arrangement of hyperplanes. The results identify the linear resolution property for the Orlik-Solomon ideal of the arrangement and the complete intersection and Koszul property of their Orlik-Solomon and Orlik-Terao ideals. 

On the other hand, a comprehensive study of the combinatorics of the broken circuits of matroids is done \cite{B}. In this regard, we have generalized a fundamental argument of \cite{B} on the decomposition of Matroids based on a lower bound for the components of the $f$-vectors of shellable complexes. The lower bound in its extreme case implies a decomposition of the complex of independent subsets of a matroid into the direct sum of two uniform matroids; see also \cite{BZ}. Our approach in generalizing the above decomposition is through the stratification of a matroid by its submatroids. The concept has been explored in \cite{CDMS}; however, it has not worked out very much in the literature. Therefore, we have explained it through basic operations on matroids \cite{CDMS, W1, W2, W3}. The stratification describes a decreasing filtration by submatroids and results in a kind of \textit{graded} decomposition on the matroid, i.e., decomposition on each stratum. 

A comprehensive exposition on Koszul algebras and the related conjectures is presented in \cite{CNR}. The characterization of this property for Orlik-Terao Ideals of arrangement of hyperplanes can be considered as a special case of having general linear resolution mentioned above. In this case, the Orlik-Terao ideal has a 2-linear resolution, \cite{RV}. In this regard, the combinatorial identification of the Koszul property for algebras associated with matroids and hyperplane arrangements is one of the significant and crucial challenging problems in this area \cite{F, CNR, Sch, RVi, SY}.   

\subsection{Organization of text}
Section \ref{sec:prel} contains the basic definitions and theorems that we use throughout the text. We have divided that into six subsections containing different notions in combinatorial algebra. Sections \ref{sec:main_res} and \ref{sec:applications} consists of our main results and contributions. Specifically, Section \ref{sec:applications} consists of the applications of the results in Section \ref{sec:main_res} to the hyperplane arrangements. In the end, we also prove a complete intersection property for the broken circuit complex of $r$-cyclic graph. Finally, in Section {sec:conc}, we present some remarks and conclusions.
\section{Preliminaries}\label{sec:prel}
\subsection{Linear resolutions} 
Let $A=K[x_1, \dots ,x_n]$ be the polynomial ring over an infinite field $k$ and $\mathfrak{m}=(x_1, \dots ,x_n)$ the maximal ideal. A finitely generated graded $A$-module $M$ is said to have \textit{$s$-linear resolution} if the graded minimal free resolution of $M$ is of the form 
\begin{equation}
0 \longrightarrow A(-s-p)^{\beta_p} \longrightarrow \  \dots \ \longrightarrow A(-s-1)^{\beta_1} \longrightarrow A(-s)^{\beta_0} \longrightarrow M \longrightarrow 0.
\end{equation}
In other words $M$ has $s$-linear resolution if $\beta_{i,i+j} =0$ for $j \ne s$; a relevant notion is that of $s$-linear quotient (see \cite{HH, AAN}). 

Next, we give some definitions and nomenclature to set the conditions of Theorem \ref{thm:linearity} presented in the following. First, let assume that $<$ is a monomial order on $A$. Write a non-zero polynomial in $A$ as, $f = \sum_v a_vv$, where $v \in A$ is a monomial. The initial monomial of $f$ w.r.t. $<$ is the biggest monomial w.r.t. $<$ among the monomials belonging to $supp(f)$. We write $in_<(f)$ for the initial monomial of $f$ with respect to $<$. The leading coefficient of $f$ is the coefficient of $in_<(f)$. If $I$ is a nonzero ideal of $A$, the initial ideal of $I$ w.r.t $<$ is the monomial ideal of $A$, which is generated by $\{in_<(f): f \in I \}$. We write $in_<(I)$ for the initial ideal of $I$. Thus, we are ready to present the following theorem, which characterizes the linear resolution property.
\begin{theorem}[\cite{RV}] \label{thm:linearity}
Assume $I= \bigoplus_k I_k$ is a graded Cohen-Macaulay ideal in $A$ of codimension $q$ and $s$ is the smallest index that $I_s \ne 0$. The followings are equivalent
\begin{itemize}
\item $I$ has $s$-linear resolution.
\item If $a_1, \dots , a_{n-q}$ is a maximal $(A/I)$-regular sequence of linear forms in $A$. we have $\overline{I}=\overline{\mathfrak{m}^s}$ in $A/(a_1, \dots ,a_{n-q})$.
\item The Hilbert function of $I$ has the form $H(I,s)={s+q-1 \choose s}$.
\item Suppose $< $ is a monomial order on $A$. If $\ in_{<}(I) \ $ (the initial ideal of $I$) is Cohen-Macaulay, then $\ in_{<}(I) \ $ has $s$-linear resolution. 
\end{itemize}
\end{theorem}
Now let us suppose that $I$ is a graded ideal of $A$ and $in_<(I)$ is Cohen-Macaulay, then from the above Theorem, it follows that $I$ has linear resolution if and only if $in_<(I)$ has a linear resolution, cf. \cite{RV} loc. Cit. A relevant definition is that of linear quotients as follows. 
\begin{definition} \label{def:colon-ideals}
Say $M$ has linear quotient, if there exists a minimal system of generators $a_1, \dots ,a_m$ of $M$ such that the colon ideals
\begin{equation} 
J_l=(a_1, \dots ,a_{l-1}):a_l
\end{equation}
are generated by linear forms of all $l$. 
\end{definition}
On the other hand, the following result relative to linear quotients is also well known.
\begin{theorem} (\cite{HH} prop. 8.2.1) \label{thm:linear-quotient}
Suppose that $I$ is a graded ideal of $A$ generated in degree $s$. If $I$ has linear quotients, then it has a $s$-linear resolution.
\end{theorem}
Another important concept related to linearity is \textit{the componentwise linearity}, \cite{HH}. For that, we consider the following construction. Assume $M=\bigoplus_j M_j$ is a finitely generated graded $A$-module and let $(\mathfrak{F}, d)= \left ( (\mathfrak{F}_n), d_n:\mathfrak{F}_n \to \mathfrak{F}_{n-1} \right )$ be the minimal graded free resolution of $M$. Define subcomplexes 
\begin{equation} 
\mathfrak{F}^{(i)}:=\big (\ \mathfrak{m}^{i-n}\mathfrak{F}_n, \ d_n \ \big ).
\end{equation}
The linear part of $\mathfrak{F}$ is the complex 
\begin{equation}
\mathfrak{F}^{\text{lin}}=\big (\ \bigoplus_i \mathfrak{F}^{(i)}/\mathfrak{F}^{(i+1)},\ gr^{(i)}(d)\ \big )
\end{equation}
Also, we have an important definition and theorem regarding the graded module $M=\oplus_j M_j$. These are presented next.
\begin{definition}
The graded module $M=\oplus_j M_j$ is said to be componentwise linear if $M_{\langle s \rangle}$ has linear resolution for all $s$, where by definition $ M_{\langle s \rangle}$ is the submodule generated by $M_s$. 
\end{definition}
\begin{theorem}[\cite{NR}, \cite{CNR}] \label{thm:componentwise-lin}
The followings are equivalent
\begin{itemize}
\item $M$ is componentwise linear.
\item $M_{\langle q \rangle}$ has a linear resolution for every $q$.
\item $\mathfrak{F}^{\text{lin}}$ is acyclic.
\end{itemize}
\end{theorem}
One can choose a homogeneous basis for the $A$-graded modules $\mathfrak{F}^{(i)}$, (for all $i$) so that matrices present the differentials with homogeneous entries. Then, the differential of the complex $\mathfrak{F}^{\text{lin}}$ is obtained by killing all the monomials of degrees bigger than one in the entries. The resulting matrix is called \textit{the linearization of differentials}.
\subsection{Broken circuit complexes}
One can consider the above properties for certain ideals constructed from combinatorial data of a matroid. Let $(X, <)$ be an ordered simple matroid of rank $r$ on $[n]$. The minimal dependent subsets of $X$ are called circuits. Denote by $BC_<(X)$ the broken circuit complex of $X$, i.e., the collection of all subsets of $M$ not containing any broken circuit. It is known that, $BC_<(X)$ is an $(r-1)$-dimensional shellable complex. The Stanley-Reisner ideal of $BC_<(X)$ denoted $I_<(X)$ is the ideal generated by all monomials on broken circuits. We denote a broken circuit by $bc_<(C) =C \smallsetminus min_<(C)$, \cite{RV}. By definition, 
\begin{equation} 
I_<(X)=\big \langle x_{bc_<(C)}:=\prod_{j \in bc_<(C)}x_j \big \rangle.
\end{equation}
The shellability of $BC_<(X)$ implies that $A/I_<(X)$ is Cohen-Macaulay. Thus, the following theorem characterizes the linear resolution of the ideal $I_<(X)$.
\begin{theorem}[\cite{RV}] \label{thm:decomposit-lin}
The following are equivalent
\begin{itemize}
\item $I_<(X)$ has $s$-linear resolution.
\item $X =U_{s,n-r+s} \bigoplus U_{r-s,r-s}$.
\item All powers of $I_<(X)$ have linear resolutions.
\end{itemize}
\end{theorem}
The matroid $U_{p,n}$ is the uniform matroid on the $n$-element, whose independent sets are subsets having at most $p$ elements. $U_{p,n}$ has rank $p$ and its circuits have $p+1$ elements. In particular, $U_{n,n}$ has no dependent sets and is called free. The proof of the Theorem \ref{thm:decomposit-lin} uses [theorem 7.5.2 of reference \cite{B}] on an external bound for the $f$-vectors; please refer to \cite{RV}. For the sake of completeness, we briefly recall this fact from \cite{B} in the following theorem.
\begin{theorem}[\cite{B}]\label{thm:inequality-fvector}
Let us consider the following.
\begin{itemize} 
\item[(1)] Let $\Delta$ be an $r-1$ dimensional shellable complex on the vector set $S$ of size $|S|=n$. If for some $s \in \mathbb{N}$, all $s$-element subsets of $S$ belong to $\Delta$, then the inequality
\begin{equation} \label{eq:inequality}
f_{k-1}\geq \sum_{i=0}^{s-1} {n-r+i-1 \choose i}  {r-i  \choose k-i}, \qquad k=0, \dots ,r  
\end{equation}
holds for the $f$-vector. 
\item[(2)] Assume $\Delta=In(X)$ is the complex of independent subsets of a loopless matroid. Then, inequality \eqref{eq:inequality} holds for $f_k$ replaced with the $I_k$, the number of independent $k$-element subsets of $X$. The equality holds if and only if 
\begin{equation}
X =U_{s,n-r+s} \bigoplus U_{r-s,r-s}, \qquad s+1=\min \{|C|\ ; C \ \text{is a circuit}\}.
\end{equation}
\end{itemize}
\end{theorem}

We sketch some explanations related to the above theorem inspired from \cite{B}. The condition that all $s$-element subsets of $S$ are in $\Delta$, implies that $f_k={n \choose k},\ k=0,1,2,3, \dots ,s$. By using formal combinatorial identities, 
\begin{equation}
h_k=\sum_{i=0}^k(-1)^{i+k}{r-i \choose k-i}{n \choose i}={n-r+k-1 \choose k} , \qquad k \leq s .
\end{equation}
Therefore, the inequality in \eqref{eq:inequality} follows by singling out the first $s+1$ terms in  
\begin{equation}
f_k=\sum_{i=0}^kh_i  {r-i \choose k-i} , \qquad k=1, \dots ,r .
\end{equation}
In case $\Delta=In(X)$, let $b(X)=I_r$, then $b(X)={n-r+s \choose s}$ is the number of bases of $X$. The $h$-vector is given by 
\begin{equation} 
\underline{\textbf{h}}=\left (\ 1, n-r,\ \dots \ , {n-r+s-1 \choose s}\ \right )
\end{equation}
and the Tutte polynomial of $X$ is 
\begin{equation}
T_X(x,1)=h_{\Delta}(x)=x^r+(n-r)x^{r-1}+ \cdots + {n-r+s-2 \choose s-1}x^{r-s}.
\end{equation}
The form of the Tutte polynomial shows that, $X$ has exactly $n-r+s$ coloops. This forces $X=N \oplus U_{r-s,r-s}$ where $N$ is a matroid of rank $s$ on $n-r+s$ points. Now $b(N)=b(X)= {n-r+s \choose s}$. This forces that $N=U_{s,n-r+s}$. The equality in \eqref{eq:inequality} holds, if $h_c=h_{c+1}= \dots =0$, cf. \cite{B} loc. cit. 

\subsection{Complete intersection property}
In Section \ref{sec:applications}, we need to decide whether a graded ideal is a complete intersection. We recall some definitions from commutative algebra. For a commutative Noetherian local ring $A$, the depth of $A$ (the maximum length of a regular sequence in the maximal ideal of $A$), is at most the Krull dimension of $A$. The ring $A$ is called \textit{Cohen–Macaulay} if its depth is equal to its dimension. A regular sequence is a sequence $a_1, \dots a_n$ of elements of $\mathfrak{m}$ such that, for all $i$, the element $a_i$ is not a zero divisor in $A/(a_1, \dots a_n)$. A local ring $A$ is a Cohen–Macaulay ring if there exists a regular sequence $a_1, \dots a_n$ such that the quotient ring $A/(a_1, \dots a_n)$ is Artinian. In that case, $n=\dim (A)$. More generally, a commutative ring is called Cohen–Macaulay if it is Noetherian, and all of its localizations at prime ideals are Cohen–Macaulay. Besides, we say that a ring $A$ is a complete intersection if there exists some surjection $R \to A$, with $R$ a regular local ring, such that a regular sequence generates the kernel. Our goal is to identify the complete intersection property for the Stanley-Reisner ideal of broken circuit ideals of matroids. For that, let us present some definitions next. 

Let us assume that $(X,<)$ is an ordered matroid on $[n]$, and $\mathcal{C}(X)$ is the set of broken circuits of $X$. Let $I_<(X)$ be the Stanley-Reisner ideal of the broken circuit complex $BC_<(X)$. A subset $D \in \mathcal{C}(X)$ is called generating set if $\langle x_{bc_<(C)}; \ C \in D \rangle$ generates $I_<(X)$. It is clear from definition that $D$ is a generating set of $\mathcal{C}(X)$ iff $bc_<(C)$ for $C \in D$ containing minimal broken circuits of $X$. Denote by $L(D)$ the intersection graph of $D$, that is the vertex set of $L(D)$ is $D$ and the edges are pairs $(C,C')$ such that $C \cap C'= \emptyset$. Say $D$ is simple if whenever $C, C' \in D$ then either $C \cap C'=min_<(C)\ \text{or} \ min_<(C')$, and otherwise $C \cap C'=\emptyset$. 
\begin{theorem}[\cite{RV}] \label{thm:complete-int}
Assume $(X,<)$ be and simple ordered matroid. The following are equivalent.
\begin{itemize}
\item $I_<(X)$ is complete intersection.
\item The minimal broken circuits of $X$ are pairwise disjoint.
\item There exists a simple subset $D \subset \mathcal{C}(X)$ such that $\mathcal{C}(X)=\{L(D');\ D' \subset D\}$.
\end{itemize}
\end{theorem}
Theorem \ref{thm:complete-int} gives a concrete combinatorial criteria to check the complete intersection property of Stanley-Reisner ideal of broken circuit complex of matroids. It can also be applied to graph matroids, see \cite{RV}.
\subsection{Hyperplane arrangements}
To an arrangement of hyperplanes $\mathcal{A}$ one naturally associates a matroid denoted by $X=X(\mathcal{A})$ with the ground set to be the same as $\mathcal{A}$ and the independent sets to be the independent hyperplanes in $\mathcal{A}$, i.e., those whose corresponding reflections as said above are independent. Assume that the matroid $X(\mathcal{A})$ is associated to an essential central hyperplane arrangement of $n$ hyperplanes $\mathcal{A}=\{H_1, \dots ,H_n\}$ in a vector space $V$ of dimension $r$ over the infinite field $k$. When $k=\mathbb{C}$, then the integral cohomology of the complement $V \smallsetminus \bigcup_i H_i$ is isomorphic to the Orlik-Solomon algebra of $A(\mathcal{A})$, which is the quotient of standard graded exterior algebra on $\mathbb{C}\langle e_1, \dots ,e_n \rangle$ by the Orlik-Solomon ideal $J(\mathcal{A})$ generated by the elements
\begin{equation}
\partial e_{i_1 \dots i_p}  =\ \sum_r\  (-1)^{r-1}e_{i_1} \dots \widehat{e_{i_r}} \dots e_{i_p},
\end{equation}
when $\{H_{i_1}, \dots ,H_{i_p}\}$  are dependent subsets of $\mathcal{A}$, i.e. the reflections $\alpha_{i_j}$ through $H_{i_j}$ are independent. Therefore,
\begin{equation}
A(\mathcal{A})=\frac{A}{J(\mathcal{A})}=\frac{\bigwedge \mathbb{Z}\ \langle e_1, \dots ,e_n \rangle}{\langle \partial e_{i_1 \dots i_p}\rangle} =H^*(V \smallsetminus \bigcup_i H_i, \mathbb{Z}).
\end{equation}
The Orlik-Terao algebra is a commutative analogue of Orlik-Solomon algebra. That is, instead of exterior algebra on $n$ vectors we work on $A=k[x_1, \dots ,x_n]$ and whenever we have a relation $\sum_j a_j\alpha_{i_j}=0$ we assume a relation $r=\sum_j a_j \alpha_{i_j}$ in a relation space $F(\mathcal{A})$. The Orlik-Terao ideal is generated by 
\begin{equation}
\partial {r} =\ \sum_j \ (-1)^{j-1}a_j x_{i_1} \dots \widehat{x_{i_j}} \dots x_{i_p}, \qquad \Leftarrow \qquad \left ( r=\sum_j a_j \alpha_{i_j} \right ).
\end{equation}
The quotient $C(\mathcal{A})=A/I(\mathcal{A})$ is called the Orlik-Terao algebra of the hyperplane arrangement. The Orlik-Solomon algebras and Orlik-Terao algebras of hyperplane arrangements are related to each other by an idea of deformation \cite{HH}. A relevant notion in this context is that of a \textit{Koszul algebra} defined next.
\begin{definition}[\cite{RV}] \label{def:koszul}
Suppose $L$ is either an exterior or a polynomial algebra over a field $k$. A graded algebra $B=L/I$ is called Koszul if $k$ has linear resolution over $B$.
\end{definition} 
The characterization of the Koszulness of the Orlik-Solomon algebra of hyperplane arrangement is one of the essential questions in this area. In this regard, we mention the following theorem.
\begin{theorem}[\cite{RV}] \label{thm:koszul-hyperplane}
Let $\mathcal{A}$ be an essential central arrangement of $n$ hyperplanes in dimension $r$, and assume $I(\mathcal{A})$ is complete intersection. Then the following are equivalent.
\begin{itemize}
\item[(1)] $A(\mathcal{A})$ is Koszul.
\item[(2)] $C(\mathcal{A})$ is Koszul. 
\item[(3)] $X(\mathcal{A})=U_{2,n-r+2} \bigoplus U_{r-2,r-2}$.
\item[(4)] $\mathcal{A}=\mathcal{A}_1 \times \mathcal{A}_2$ where $\mathcal{A}_1$ is generic central arrangement of $n-r+2$ hyperplanes in dimension $p$ and $\mathcal{A}_2$ is Boolian in dimension $r-2$.
\item[(5)] $I(\mathcal{A})$ has $2$-linear resolution.
\item[(6)] $J(\mathcal{A})$ has $2$-linear resolution.
\end{itemize}
\end{theorem}
It is well-known that if $B$ is Koszul then $I$ is generated by quadrics. Unfortunately, the converse is not true in general. The following consequence is immediate from the above theorem. %
\begin{theorem}[\cite{RV}] \label{thm:hyperplane-arrange}
The Orlik-Terao ideal $I(\mathcal{A})$ of an essential central arrangement has 2-linear resolution if $\mathcal{A}$ is obtained by successively coning a central arrangement of lines in the plane. In this case The Orlik-Terao algebra is Koszul.
\end{theorem}

\subsection{Basic operations on matroids}
Assume $X$ is a matroid with an element set $E$, and $S$ is a subset of $E$. The restriction of $X$ to $S$, written $X | S$, is the matroid on the set S whose independent sets are the independent sets of $X$ that are contained in $S$. Its circuits are the circuits of $X$ contained in S, and its rank function is that of $X$ restricted to subsets of $S$. In linear algebra, this corresponds to restricting to the subspace generated by the vectors in $S$. If $T= X-S$, this is equivalent to the deletion of $T$, written $X\smallsetminus T$. The submatroids of $X$ are precisely the results of a sequence of deletions (the order is irrelevant). The dual operation of restriction is contraction. If $T$ is a subset of $E$, the contraction of $X$ by $T$, written $X/T$, is the matroid on the underlying set $E-T$ whose rank function is ${\displaystyle r'(A)=r(A\cup T)-r(T)}$. In linear algebra, this corresponds to looking at the quotient space by the linear space generated by the vectors in $T$, together with the images of the vectors in $E-T$. Then, the Tutte polynomial is defined by the recurrence relation ${\displaystyle T(A)=T(A\cup T)-T(T)}$, \cite{W1, W2, W3}. Given a matroid $([n], I)$, and if $X \subset [n]$ then, the deletion $[n] \smallsetminus X$ is the matroid with ground set $[n] \smallsetminus X$ and independent sets $\{J \subset [n] \smallsetminus X : J \in I\}$. The matroid dual $X^*$ of $X$ is the matroid on $[n]$
where $I$ is a basis of $X$ iff $I^c$ is a basis of $X^*$. If $X \subset [n]$ then the contraction can be  defined as $N/X = (N^* \diagdown X)^*$. 

Now, let $X$ and $Y$ be matroids on the same ground set $E$. We say
that $Y$ is a quotient of $X$ if one of the following equivalent statements holds:
\begin{itemize}
\item Every circuit of $X$ is a union of circuits of $Y$.
\item If $X' \subset Y' \subset E$, then $r_M(Y' )-r_M(X') \geq r_N (Y')-r_N (X')$, ($r$ is the rank function, see \cite{CDMS}).
\item There exists a matroid $R$ and a subset $X'$ of $E(R)$ such that $X=R \smallsetminus X'$ and $Y = R/X'$.
\item For all bases $B$ of $X$, and $x \in B$, there is a basis $B_0$ of $Y$, with $B_0 \subset B$, such that $\{y: (B_0 -y) \cup x \in B(N)\} \subset \{y : (B -y) \cup x \in B(X)\}$.
\end{itemize}
For the equivalences, we refer to \cite{O}, Proposition 7.3.6. If $Y$ be a quotient of $X$, it follows that:
\begin{itemize}
\item Every basis of $Y$ is contained in a basis of $X$, and every basis of $X$ contains a basis of $Y$.
\item $rk(Y) \leq rk(X)$ and in case of equality, $Y = X$.
\end{itemize}
cf. \cite{CDMS} loc. cit. 
\begin{definition} \label{def:stratification}
A stratification of a matroid $X$ is a decreasing filtration sequence $X=M_0 \supset X_1 \supset \dots \supset X_d$ by submatroids $X_j$, where we call $X_j \smallsetminus X_{j+1}$ the $j$-th strata of the stratification. 
\end{definition} 
The Tutte polynomial of $X$ becomes the sum of the Tutte polynomials on each strata $S_j=X_j \smallsetminus X_{j+1}$, [see \cite{W1, W2, W3}]. 
\subsection{The r-cyclic graph} \label{subsec:graph}
Let $(G,V,E)$ be a graph, $V$ the vertices and $E$ the edges. Set $\mathcal{C}$ the edge set of cycles in $G$. It is the set of circuits of a matroid, namely \textit{cycle matroid}. We assume the graph and this matroid are simple, i.e., every circuit has at least three vertices. In \cite{AR} the Cohen-Macaulayness of the facet ideal of $r$-cyclic graph $G_{n,r}$ has been discussed. We can show that the facet ideal of $G_{n,r}$ is a complete intersection.
\begin{definition}[\cite{AR}] \label{def:cyclic graph}
Let $\Delta$ be a simplicial complex over $[n]$. Let $I_F(\Delta)$ be the monomial ideal minimally generated by square-free monomials $f_{F(1)}, \dots ,f_{F(s)}$ such that 
\begin{equation} 
f_{F(i)} =x_{i_1}x_{i_2} \dots x_{i_r}, \qquad \Leftrightarrow \qquad F(i) =\{i_1, \dots ,i_r\} \ \text{is a facet of} \ \Delta \ \text{ for all} \ i	,
\end{equation}
known as the facet ideal of complex $\Delta$ and denoted $I_F(\Delta)$. The graph $G_{n,r}$ has $n$ edges and $r$ distinct cycles. One naturally associates a simplicial complex $\Delta(G_{n,r})=\langle E \smallsetminus \{e_{1i_1}, \dots ,e_{ri_r} \}\rangle$ to $G_{n,r}$, where $C_i=\{ e_{1_1}, \dots ,e_{1_i}\}$ is the $i$-th cycle, \cite{AR}.
\end{definition}
The theory of matroids and their broken circuit complexes is very much tied with the theory of graphs. In Section \ref{sec:applications} we present an example of this feature in our results.
%
%
%
\section{Main Results}\label{sec:main_res}
In this section, we extend some main results presented in \cite{RV} to the corresponding graded notions. All the definitions and theorems in this section are new results. We prove three kinds of results: a) The first is the definition of the graded linearity and graded linear quotient property. We compare the two notions of graded linearity and graded linear quotients on graded $A$-modules, and then we also compare graded linearity with componentwise linearity. b) The second is to identify the graded linearity of the Stanley-Reisner ring of broken circuit complexes in terms of the Hilbert polynomial. c) The third is a stratified decomposition property of the matroid when the Stanley-Reisner ideal of broken circuit complex has graded linear resolution.
\subsection{Graded linear resolutions}
We introduce the two following definitions. Next, we prove a theorem that compares the two definitions.
\begin{definition} 
We say $M$ has graded linear quotient, if there exists a minimal system of generators $a_1, \dots ,a_m$ of $M$ such that the colon ideals satisfy $indeg(J_l)=1$ \footnote{We call the least degree of a homogeneous generator of a graded $A$-module $M$, the
initial degree of $M$, denoted by $indeg(M)$.}.
\end{definition}
\begin{definition}
Say $M$ has graded linear resolution if the graded Betti numbers in the graded Betti table appear in a finite number of horizontal lines, with finitely many non-zero entries in each line, and successively.  
\end{definition}
An example of having graded linear resolution is that of having linear resolution. However, the graded linearity is more general. We stress that having linear property for the graded modules under consideration is a strictly special case. Then, we present our next main result.
\begin{theorem}[Main result]
Assume that $M$ is a graded $A$-module minimally generated by a regular sequence $(a_1, \dots ,a_n)$ in $\mathfrak{m}$; and $J_l$ is defined as in Definition \ref{def:colon-ideals}. If $M$ has graded linear quotients, then $J_l$ has graded linear resolution for any $l$.
\end{theorem}
\begin{proof}
We use induction on $l$. Consider the short exact sequence, 
\begin{equation}
0 \to \frac{A}{(J_l : z_{l+1})} \stackrel{\times z_{l+1}}{\longrightarrow} \frac{A}{J_l} \longrightarrow \frac{A}{J_{l+1}} \to 0. 
\end{equation}
The first and the last module have quasi-linear resolutions by induction assumption. Now, consider the associated long graded Tor exact sequence:
\begin{align}
\longrightarrow {Tor_i\Big (\frac{A}{(J_l : z_{l+1})},k\Big )}_j \stackrel{\times z_{l+1}}{\longrightarrow} {Tor_i\Big (\frac{A}{J_{l}},k\Big )}_{j+d} \longrightarrow {Tor_i\Big (\frac{A}{J_{l+1}},k\Big )}_{j+d} \longrightarrow \\ {Tor_{i-1}\Big (\frac{A}{(J_l : z_{l+1})},k\Big )}_{j+d-1} 
 \stackrel{\times z_{l+1}}{\longrightarrow} {Tor_{i-1}\Big (\frac{A}{J_{l}},k\Big )}_{j+d-1} \longrightarrow  
\end{align}
where $d=\deg(a_{l+1})$. One notes that multiplication by $z_{l+1}$ is in the annihilator of $k$. This implies 
the sequence    
\begin{equation}
0 \to {Tor_i\Big (\frac{A}{J_{l}},k \Big )}_{j+d} \longrightarrow {Tor_i \Big (\frac{A}{J_{l+1}},k\Big )}_{j+d} \longrightarrow {Tor_{i-1}\Big (\frac{A}{(J_l : a_{l+1})},k \Big )}_{j+d-1} \to 0
\end{equation}
is short exact. We note that ${Tor_i\Big (\frac{A}{J_{l}},k \Big )}_{j+d}$ embeds into ${Tor_i \Big (\frac{A}{J_{l+1}},k\Big )}_{j+d}$. Therefore, the non-triviality of the first one implies the non-triviality of the second. Therefore, in order to check the induction step it suffices to prove
\begin{equation} 
Tor_i \big (\frac{A}{J_{l+1}},k \big )_{j+d} \ne 0 \ \iff \ Tor_{i-1} \big (\frac{A}{(J_l : a_{l+1})},k \big )_{j+d-1}  \ne 0.
\end{equation}
The last assertion is true for the implication
\begin{equation}
{Tor_{i-1}\Big (\frac{A}{(J_l : a_{l+1} )},k\Big )}_{j+d-1}  =0  \qquad   \Rightarrow   \qquad  {Tor_i\Big (\frac{A}{J_{l}},k\Big )}_{j+d} =0.
\end{equation}
Now, the assertion follows by induction, and the proof is complete. 
\end{proof}
In the following theorem, we compare the graded linearity with the componentwise linearity defined in the previous section. As a result, the graded linear property is weaker than componentwise linearity. Therefore, our definition can be regarded as a more general approximation to the concept of componentwise linear resolution.
\begin{theorem}[Main result] 
Assume $M$ is a finitely generated graded $A$-module. Consider the following statements
\begin{itemize}
\item[(1)] $M$ has graded linear resolution.
\item[(2)] $M_{\langle s \rangle}$ have graded linear resolutions for all $s$.
\item[(3)] $M$ is componentwise linear.
\end{itemize}
Then, we have (3) $\Rightarrow$ (2) $\Leftrightarrow$ (1).
\end{theorem}
\begin{proof}:
The implications (3) $\Rightarrow$ (2) $\Rightarrow$ (1), are evident by Theorem \ref{thm:componentwise-lin}. Therefore, we check out (1) $\Rightarrow$ (2). This follows from the fact that a graded resolution of $M_{\langle s \rangle}$ can be included (is a subresolution) in a graded resolution of $M$.
\end{proof}

\subsection{Graded linearity and broken circuit complexes}
We want to characterize the graded linear resolution property for the Stanley-Reisner rings of broken circuit complexes. The following result generalizes Theorem \ref{thm:decomposit-lin} presented in the previous section and conforms to part of our main results.
\begin{theorem} \label{thm:grlin-StanleyReisner}
Let $(X, <)$ be an ordered simple matroid of rank $r$ on $[n]$. The followings are equivalent:
\begin{itemize}
\item[(1)] $I_<(X)$ has graded linear resolution.
\item[(2)] If $a_1, \dots , a_{n-q}$ is a maximal $(A/I_<(X))$-regular sequence of linear forms in $A$, we have 
\begin{equation} \label{eq:graded-quotient}
\frac{\overline{A}}{\overline{I_<(X)}}=\bigoplus_{i=0}^d\ \frac{\overline{\mathfrak{m}}^{s_i}}{\overline{\mathfrak{m}}^{s_{i+1}}},
\end{equation}
in $A/(a_1, \dots ,a_{n-s})$.
\item[(3)] The Hilbert function of $I_<(X)$ has the form,
\begin{equation}
H(\frac{A}{I_<(X)},s)=c_0{s+q-1 \choose s}+c_1{s+q-2 \choose s}+ \cdots +c_{d-1}{s+q-d \choose s},
\end{equation}
for some $d$.
\item[(4)] If $\prec $ is a monomial order on $A$ and if $\ in_{\prec}(I_<(X))\ $ is Cohen-Macaulay, $\ in_{\prec}(I_<(X))\ $ has graded linear resolution.
\end{itemize}
\end{theorem}
\begin{proof} (3) follows from (2), and also, (4) follows from (3) and the equality of the Hilbert functions 
\begin{equation} 
H(A/I_<(X),s)=H(\overline{A/I_<(X)},s),
\end{equation}
where the bar means the image in $A/(a_1, \dots ,a_{n-q})$. Also, one has the stronger statement 
\begin{align}
H(A/I_<(X),s)=H(\overline{A/I_<(X)},s)\leq H(\bigoplus_{i=0}^d\ \frac{\overline{\mathfrak{m}}^{s_i}}{\overline{\mathfrak{m}}^{s_{i+1}}},s)=\sum_{i=0}^{d-1}c_i{s+q-i-1 \choose s},
\end{align}
with equality when \eqref{eq:graded-quotient} holds. In fact, by factoring out the regular sequence $(a_1, \dots ,a_{n-q})$ one can assume $A$ is Artinian. Then, it is easy to establish that it has the form \eqref{eq:graded-quotient}, [see also \cite{RV} proposition 3.1]. This also proves (2).  
\end{proof}
\begin{remark}
The coefficients $c_i=c_i(I_<(X))$ can be easily calculated, (called Hilbert coefficient of $I_<(X)$). The sum in the Hilbert polynomial is the sum of the Hilbert function of each summand where can be calculated by hand. 
\end{remark}
The following lemma is a generalization of proposition 7.5.2 in \cite{B} and forms part of our main result. 
\begin{lemma} \label{thm:lemma}
Assume $X$ is a rank $r$ loopless matroid on $[n]$ such that the Stanley-Reisner ideal of the broken circuits of $X$ has graded linear resolution. Then,
\begin{equation} \label{eq:gen-inequality}
I_{j-1} \geq \sum_{i=0}^{c-1}\sum_{l=1}^dc_l^{(S)}{n-r+l-1 \choose l}{r-i \choose j-i} 
\end{equation}
and the equality holds if and only if there exists a stratification $X=X_0 \supset X_1 \supset \dots \supset X_d$ by submatroids $X_j$ for some $d$ such that on each strata,
\begin{equation}
X_j \diagdown X_{j+1}= U_{r_j-s_j,r_j-s_j} \bigoplus U_{s_j,n_j-r_j+s_j},
\end{equation}
where $s_j \leq r_j$ are positive integers and $U_{l,k}$ denotes the uniform matroid. And also, We have $n=\Sigma_j n_j, \ r=\Sigma_j r_j$. 
\end{lemma}
\begin{proof} 
Following the argument of Theorem \ref{thm:inequality-fvector} of previous section [or theorem 7.2.5 \cite{B}], and letting $\Delta=In(X)$, then we have,
\begin{equation} 
h_k=H(A/I_<(X),k)=\ \sum_lc_l{k+q-l-1 \choose k}, \qquad k=0,1,2, \dots ,c-1,
\end{equation} 
where $c_l$ are Hilbert coefficients of $I_<(X)$. Then,
\begin{equation}
b(X)=\sum_lc_l{q-l+p \choose p}     
\end{equation}
is the number of bases of $X$. Besides, the $h$-vector is given by,
\begin{equation} 
\underline{\textbf{h}}=\left (\ 1 \ \ ,\  \sum_lc_l{q-l \choose 1},\ \  \dots \ \ ,\ \sum_l c_l{c-1+q-l-1 \choose c-1}\ \right ),
\end{equation}
and the Tutte polynomial of $X$ is $T_X(x,1)=h_{\Delta}(x)=x^r+h_1x^{r-1}+ \cdots + h_px^{r-p}$. The form of the Tutte polynomial shows the existence of stratification $X=X_0 \supset X_1 \supset \dots \supset X_d$ by submatroids $X_j$ such that,
\begin{equation}
X_j \diagdown X_{j+1}= U_{r_j-s_j,r_j-s_j} \bigoplus N_j
\end{equation}
holds, where $N_j$ is a matroid of rank $s_j$ on $n_j-r_j+s_j$ points. In fact, the $\underline{\textbf{h}}$-vector of the submatroids $X_j$ is given by the truncations of the sums appearing in the $\underline{\textbf{h}}$-vector of $X$. Thus,
\begin{equation}
b(N_j)=b(X_j \diagdown X_{j+1})=  c_j{c-j+q-l \choose c-1}  
\end{equation} 
holds, and this shows that $N_j$ has the desired form. The equality in \eqref{eq:gen-inequality} holds when $h_c=h_{c+1}= \dots =0$.
\end{proof}
We can now formulate the graded linearity of the Stanley-Reisner ideal of the broken circuits in terms of the combinatorics of $M$. The Theorem \ref{thm:grlin-StanleyReisner} and the Lemma \ref{thm:lemma} can be used together to obtain the following decomposition property, expressed in the following result.
\begin{theorem}[Main result]\label{thm:stratification}
Let $(X,<)$ be an ordered simple matroid of rank $r$ on $[n]$, and $I_<(X)$ the Stanley-Reisner ideal of broken circuit complex of $X$. Then, the followings are equivalent.
\begin{itemize}
\item[(1)] $I_<(X)$ has graded linear resolution.
\item[(2)] There exists a stratification $X=X_0 \supset X_1 \supset \dots \supset X_d$ for some $d$ such that on each strata 
\begin{equation}
X_j \diagdown X_{j+1}= U_{r_j-s_j,r_j-s_j} \bigoplus U_{s_j,n_j-r_j+s_j}
\end{equation}
holds, where $s_j \leq r_j$ are positive integers. Besides, we have $n=\Sigma_j n_j, \ r=\Sigma_j r_j$. 
\end{itemize}
\end{theorem}
\begin{proof}
We first show that $(2) \Rightarrow (1)$. First, assume that (2) holds. Then, by re-arrangement of variables, we can assume that
\begin{align}
I_<(X)=\big \langle \ x_{j_1} \cdots x_{j_{s_1}}\ ,\ \dots \ ,\  x_{l_1}\cdots x_{l_{p}} \ | \ \ j_1 < \dots < j_{s_1} < n-r+s_1\\
,\ \dots \ ,\  l_1 < \dots < l_{p}< n-r+p\ \big \rangle.    
\end{align}
This ideal has graded linear resolution.

Next, we show $(1) \Rightarrow (2)$. For that, let us assume that (1) holds, i.e., assume $I_<(X)$ has graded linear resolution. Then, the properties (2) and (3) of Theorem \ref{thm:grlin-StanleyReisner}, if $\textbf{a}=(a_1,\dots ,a_r)$ is a maximal $S/I_<(X)$-regular sequence of linear forms, then
\begin{equation}
S=A/(I_<(X)+\textbf{a}) = \bigoplus_{j=0}^d \overline{\mathfrak{m}}^{s_j} / \overline{\mathfrak{m}}^{s_{j+1}}, \qquad (\mathfrak{m}^0=A)  .
\end{equation}
Thus, the Hilbert function of $A/I_<(X)$ is given by
\begin{equation}
H(A/I_<(X), t)= H(S,t)/\prod_i(1-t^{s_i}) .
\end{equation}
So that the $f$-vector of Stanley-Reisner ring is given by
\begin{equation}
f_{j-1} = \sum_{i=0}^rc_i{r-i \choose j-i}  =\sum_{i=0}^r\sum_{l=1}^dc_l{n-r+l-1 \choose l}{r-i \choose j-i}, \end{equation}
where $d=\deg(H(S,t))$ and $c_l=c_l^{(S)}$ are the coefficients of $H(S,t)$. Thus, the lemma \ref{thm:lemma} implies the decomposition in item (2).
\end{proof}
\begin{remark}
We highlight three main points:
\begin{itemize}
\item There are modules over the polynomial ring that are not componentwise linear but satisfy being \textit{graded linear}. A simple example is a direct sum of a non-componentwise linear module with one componentwise linear, where the resolutions are just direct sums of the two resolutions of the modules. Being "graded linear" is a relatively weak property compared to being componentwise linear.
\item Let $X$ be the graphic matroid of parallel connection of a 3-circuit with a 4-circuit. Then, $X$ is a non-uniform 2-connected matroid, and therefore can not be a direct sum of uniform matroids. We can order the elements of $X$ from 1 to 6 in a way that the Stanley-Reisner ideal of the broken circuit complex of $M$ gets equal to $I(X)=
(x_5x_6, x_2x_3x_6, x_2x_3x_4x_5)$. Then, $I(X)$ has a linear quotient and is, therefore, componentwise linear. This example shows that the stratification mentioned in the above theorem is essential.
\item There exist componentwise (matroidal) ideals, which are not generated by the regular sequence. For instance, let $M$ be the same as the previous item, the parallel connection of a 3-circuit and a 4-circuit, then $I(X)$ is componentwise linear and has linear quotients, but a regular sequence does not generate it.
\end{itemize}
\end{remark}
\section{Applications} \label{sec:applications}
This section applies the results of the previous section to the Orlik-Solomon ideal and the Orlik-Terao ideal of the arrangements of hyperplanes. Our results are direct interpretations of the theorems presented in Section \ref{sec:main_res}, but into the language of hyperplane arrangements. All the definitions and theorems of this section are new contributions to this text.
\subsection{Application to hyperplane arrangements}
We begin by the following theorem.
\begin{theorem}\label{thm:gradedlinear-hyperplane}
For an essential central arrangement of $\mathcal{A}$ of $n$ hyperplanes, located in a vector space of dimension $r$, the following points are equivalent.
\begin{itemize}
\item $I(\mathcal{A})$ has graded linear resolution.
\item There exists a stratification $X=X_0 \supset X_1 \supset \dots \supset X_d$ by submatroids $X_j$ for some $d$ such that  
$X_j(\mathcal{A}) \diagdown X_{j+1}(\mathcal{A})= U_{r_j-s_j,r_j-s_j} \\  \bigoplus U_{s_j,n_j-r_j+s_j}$ where $s_j \leq r_j$ are positive integers. Also, we have $n=\Sigma_j n_j, \ r=\Sigma_j r_j$. 
\item $\mathcal{A}=\mathcal{A}_1 \times \mathcal{A}_2 \times \dots \times \mathcal{A}_k$ where $\mathcal{A}_1$ is boolian and the other factors $\mathcal{A}_j$ for $ j >1$ are generic arrangements of $n_j-r_j+s_j$ hyperplanes ina $s_j$-dimensional vector space. 
\end{itemize}
\end{theorem}
\begin{proof}
The proof is analogous to the 2-decomposition case in \cite{RV}, and is just an interpretation of the Theorems in the previous section.
\end{proof}
In the following, we generalize the Koszul property to that of graded Koszul in a natural way. First, let consider the following definition.
\begin{definition} 
In the settings of Definition \ref{def:koszul} we call $B$ graded Koszul if $k$ has graded linear resolution.
\end{definition} 
The following consequence is immediate from the above theorem.
\begin{theorem}
The Orlik-Terao ideal $I(\mathcal{A})$ has 2-graded linear resolution if $\mathcal{A}$ is obtained by successively coning central arrangements of lines in the plane. In this case, the Orlik-Terao algebra is graded Koszul.
\end{theorem}

\begin{proof}
The proof is a direct application of Theorem \ref{thm:gradedlinear-hyperplane}.
\end{proof}
\begin{theorem}
With the same set-up of Theorem \ref{thm:gradedlinear-hyperplane} the following are equivalent
\begin{itemize}
\item[(1)] $A(\mathcal{A})$ is graded Koszul.
\item[(2)] $C(\mathcal{A})$ is graded Koszul. 
\item[(3)] There exists a stratification $X=X_0 \supset X_1 \supset \dots \supset X_d$ by submatroids $X_j$ for some $d$ such that  
$X_j(\mathcal{A}) \diagdown X_{j+1}(\mathcal{A})= U_{r_j-s_j,r_j-s_j} \\  \bigoplus U_{s_j,n_j-r_j+t_j}$, where $s_j \leq r_j$ are positive integers. Also, we have $n=\Sigma_j n_j, \ r=\Sigma_j r_j$. 
\item[(4)] $\mathcal{A}=\mathcal{A}_1 \times \mathcal{A}_2 \times \dots \times \mathcal{A}_k$ where $\mathcal{A}_1$ is boolian and the other factors $\mathcal{A}_j$ for $ j >1$ are generic arrangement of $n_j-r_j+s_j$ hyperplanes in a $s_j$-dimensional vector space. 
\item[(5)] $I(\mathcal{A})$ has $2$-graded linear resolution.
\item[(6)] $J(\mathcal{A})$ has $2$-graded linear resolution.
\end{itemize}
\end{theorem}
\begin{proof}
The equivalence of (2)-(5) was already established above. The equivalence of (1)-(2) follows from Theorem \ref{thm:gradedlinear-hyperplane} and the argument in Theorem \ref{thm:stratification}.
\end{proof}
The characterization of the Koszul property of Orlik-Solomon algebras of matroids (arrangments) is one of the essential problems in this area. It is a conjecture that; a matroid (an arrangement) is supersolvable if and only if its Orlik-Solomon algebra is Koszul, [see \cite{RV, CNR} for instance].
\subsection{Application on the facet ideal of the $r$-cyclic graph}
Our last result is an application to prove the entire intersection property of the facet ideal of the $r$-cyclic graph, [see Subsection \ref{subsec:graph}].
\begin{proposition}
The facet ideal $I_F(\Delta(G_{n,r}))$ is a complete intersection.
\end{proposition}
\begin{proof}
By definition, $G_{n,r}$ is a graph on $n$ edges with exactly $r$ distinct cycles. Therefore, the broken circuits of the cycle matroid $X(G_{n,r})$ are pairwise disjoint. On the other hand, the Stanley-Reisner ideal of the broken circuit complex of cycle matroid $X(G_{n,r})$ is exactly the facet ideal of the complex $\Delta(G_{n,r})$. Now, the theorem follows from Theorem  \ref{thm:complete-int}.
\end{proof}
\begin{corollary}
The facet ideal $I_F(\Delta(G_{n,r}))$ is Cohen-Macaulay.
\end{corollary}
The corollary has been proved in \cite{AR} by a direct computational method. We proved it as an immediate consequence of our result. A comprehensive analysis of the entire intersection property for graph matroids related to broken circuits is given in \cite{RV}.
%
%
%
%
\section{Conclusion}\label{sec:conc}
The notion of graded linear resolution is introduced, and several characterizations of this property are shown for certain the graded algebras associated with matroids. The notion is compared with that of graded linear quotions and componentwise linearity in general. Specifically, it is proved that the graded linearity of the Stanley-Reisner ideal of broken circuit complex of a simple matroid implies a mixed decomposition of a matroid into uniform matroids, along with a stratification. The translation of these results into the language of hyperplane arrangements is also given at the end. Finally, an application to complete the intersection property of the facet ideal of $r$-cyclic graph is given.
\section*{Conflict of Interest} 
The work's authors declare that they do not have any conflict of interests. 
\section*{Author contributions statement}
M. Reza-Rahmati and G. Flores reviewed the manuscript and contributed equally to the work.

\bibliographystyle{hindawi_bib_style}

\end{document}